\documentclass[12pt]{article}

\usepackage{amsthm}
\usepackage[english]{babel}
\usepackage[utf8]{inputenc}
\usepackage{amsmath}
\usepackage{amsthm}
\usepackage{graphicx}
\usepackage{authblk}

\usepackage[colorinlistoftodos]{todonotes}
\usepackage{wasysym}
\usepackage{amsfonts}
\usepackage{ amssymb }

\title{Left-Right Pairs and Complex Forests of Infinite Rooted Binary Trees}

\author{Nina Zubrilina}
\affil{Stanford University Department of Mathematics\\ Email: nina57@stanford.edu}

\date{\today}

\begin{document}
\maketitle

\newtheorem{theorem}{Theorem}[section]
\newtheorem{fact}[theorem]{Fact}
\newtheorem{corollary}[theorem]{Corollary}
\newtheorem{lemma}[theorem]{Lemma}
\newtheorem{defn}[theorem]{Definition}
\newtheorem{remark}[theorem]{Remark}
\newtheorem*{claim}{Claim}
\newtheorem{prob}{Problem}[section]
\newcommand\underrel[2]{\mathrel{\mathop{#2}\limits_{#1}}}
\newcommand\tab[1][0.5cm]{\hspace*{#1}}
\newcommand{\Prob}{\mathbb{P}}
\newcommand{\E}{\mathbb{E}}
\newcommand{\eps}{\varepsilon}
\newcommand{\norm}[1]{\left\lVert#1\right\rVert}
\newcommand\abs[1]{\left|#1\right|}
\newcommand{\N}{\mathbb{N}}
\newcommand{\Nz}{\mathbb{N}_0}
\newcommand{\Z}{\mathbb{Z}}
\newcommand{\R}{\mathbb{R}}
\newcommand{\Hp}{\mathbb{H}}
\newcommand{\Q}{\mathbb{Q}}
\newcommand{\C}{\mathbb{C}}
\newcommand{\SLN}{\mathrm{SL}_2(\mathbb{N}_0)}
\newcommand{\G}{\mathcal{G}}
\newcommand{\fancyF}{\mathcal{F}}
\newcommand{\D}{\mathcal{D}_0}
\newcommand{\Dbar}{\overline{\mathcal{D}}_0}
\newcommand{\mat}[4] {\left(\begin{array}{cccc}#1 & #2\\ #3 & #4 \end{array}\right)}
\newcommand{\Cbar}{\overline\C}
\newcommand{\PIm}{\mathbb{I}_{\geq0}}
\newcommand{\rad}{\mathrm{rad}}
\newcommand{\diam}{\mathrm{diam}}
\newcommand{\id}{\mathrm{Id}}
\renewcommand{\L}{\mathcal{L}}
\renewcommand{\Im}{\mathrm{Im}}

\begin{abstract}
Let $\D:= \{x + iy \ \vert x, y >0\}$, and let $(L, R)$ be a pair of M\"{o}bius transformations corresponding to $\SLN$ matrices such that $R(\D)$ and $L(\D)$ are disjoint. Given such a pair (called a left-right pair), we can construct a directed graph $\fancyF(L, R)$ with vertices $\D$ and edges $\{(z, R(z))\}_{z \in \D} \cup \{(z, L(z))\}_{z \in \D}$, which is a collection of infinite binary trees. We answer two questions of Nathanson by classifying all the pairs of elements of $\SLN$ whose corresponding M\"{o}bius transformations form left-right pairs and showing that trees in $\fancyF(L, R)$ are  always rooted.

\end{abstract}

\section{Introduction}
An \emph{infinite rooted binary tree} is a directed tree with infinitely many vertices such that every vertex has outdegree $2$, there is one vertex with indegree $0$, and every other vertex has indegree $1$. In an \emph{infinite rootless binary tree} every vertex has outdegree $2$ and indegree $1$. 

Let $\Cbar = \C \cup \{\infty\}$ be the Riemann sphere, $\Nz:= \N \cup \{0\}$ and $$\SLN: = \Bigg\{\mat{a}{b}{c}{d} \ \vert \ a, b, c, d \in \N_0, ad - bc = 1\Bigg\}.$$ For $T= \mat{a}{b}{c}{d} \in \SLN$, we let $T(z): \Cbar \to \Cbar$ be the corresponding M\"{o}bius transformation:
$$T(z):= \frac{az + b}{cz + d}.$$ We let $\D: = \{x + iy \ \vert x, y > 0\} \subseteq \Cbar$, and for $L, R \in \SLN$, we call $(L, R)$ a \emph{left-right pair} if $L(\D) \cap R(\D) = \emptyset.$ For any $T \in \SLN$, $T(\D) \subseteq \D,$ so for $L, R \in \SLN$, the graph $\fancyF(L, R)$ with vertices $\D$ and edges $ \big\{\big(z, L(z)\big)\big\vert \ z \in \D\big\}\cup\big\{\big(z, R(z)\big)\big\vert \ z \in \D\big\}$ is well-defined; moreover, if $(L, R)$ is a left-right pair, $\fancyF(L, R)$ is a collection of infinite binary trees (\cite{nathanson}).

Graphs constructed this way on various domains often have interesting numerical properties. A particularly well-studied example of this is the Calkin-Wilf tree (\cite{calkinwilf}), which is a tree with vertices $\Q_{>0}$ built on the left-right pair 
$$R = \mat{1}{0}{1}{1}, \ L = \mat{1}{1}{0}{1}.$$ 
Motivated by the beautiful properties of the Calkin-Wilf tree (see \cite{bates}, \cite{bates2}, \cite{chan}, \cite{dilcher}, \cite{gibbons}, \cite{han}, \cite{Mallows}, \cite{mansour},  \cite{nath2}, \cite{nath3}, \cite{nath4}, \cite{rez}), Nathanson (\cite{nathanson}) studied the more general construction $\fancyF(L, R)$ described above. He posed the following problems:

\begin{enumerate}
\item Classify left-right pairs in $\SLN$; 
\item Determine left-right pairs $(L,R)$ whose associated forests contain infinite binary rootless trees. 
\end{enumerate}
We answer these questions with the following theorems:
\begin{theorem}\label{lrpairs}
Let $$A = \mat{a_1}{b_1}{c_1}{d_1}, B = \mat{a_2}{b_2}{c_2}{d_2} \in \SLN.$$ Then $(A, B)$ is a left-right pair if and only if $a_1d_2 \leq b_2c_1$ or $ a_2d_1 \leq  c_2b_1$.
\end{theorem}
\begin{theorem}\label{allrooted}
Let $L, R \in \SLN$ such that $(L, R)$ is a left-right pair. Then all the trees in $\fancyF(L, R)$ are rooted.
\end{theorem}

We let $\Dbar := \D \cup \partial \D = \D \cup \R_{\geq 0} \cup i \cdot \R_{\geq 0} \cup \{\infty\}$. We let $\Hp := \{x + iy \ \vert y >0\} $ be the Poincar\'e half-plane model for hyperbolic space with $\partial \Hp = \R \cup\{ \infty\}$ its boundary. We let $\G$ be the set of geodesics of $\Hp$, which are semicircles with end points on the real axis and vertical rays emitting from points on the real axis, endpoints included. Every element of $\G$ intersects the boundary of $\Hp$ in exactly two points, which we will refer to as the endpoint of $g$. M\"{o}bius transformations map elements of $\G$ into one another bijectively, with endpoints mapping to endpoints. 

For $g \in \G$ such that $g \subseteq \Dbar$, the open region contained in $\D$ and bounded by $g$ and $\R \cup \{\infty\}$ is called a \emph{slice}. We call the endpoints of $g$ the endpoints of that slice. For a slice $P$, we let 
$$\rad(P):= \sup_{z \in P} \Im(z), \diam(P) = 2 \rad(P).$$ Note that $\rad(P)$ is the radius of the geodesic bounding $P$ when the geodesic is a semicircle and infinity when it is a ray. We let $g(P)$ be the geodesic bounding $P$. If $g(p)$ has endpoints $x, y$ with $x < y < \infty$, we let $I(P) := [x, y] \subseteq \R_{\geq 0}$; if $g(P)$ has endpoints $x, \infty$, we let $I(P):= [x, + \infty )\cup \{\infty\}.$ Note that $\partial P = g(P) \cup I(P)$. 

We call $w$ an an \emph{ancestor} of $z$ if there is a way from $w$ to $z$ via the edges of $\fancyF(L, R)$.

We will abuse notation by writing $x/0 := \infty \in \Cbar$ when $x \neq 0$. We will also write $\infty > x$ for all $x \in  \R$.

\section{Left-right pairs}
In this section we classify left-right pairs. We begin with a Lemma:

\begin{lemma}\label{everything}

Let $M = \mat{a}{b}{c}{d}\in \SLN$, and let $P$ be a slice with endpoints $x,y$, with $x < y$. Then $M(P)$ is a slice with endpoints $M(x),M(y)$ and $M(x) <M(y)$. 
\end{lemma}
\begin{proof}
 If $c = 0$, since $ad - bc = 1$, we can conclude that $a = d = 1$ and $M$ is a translation by a non-negative integer $b$, so $M(P)$ is again a slice with endpoints $x + b = M(x)$, $y + b =  M(y)$.  
 
Suppose now $c \neq 0$. Let $I = I(P), g = g(P)$.  $M$ maps reals into reals, is continuous and injective and preserves orientation, so $M(I) = [M(x), M(y)]$. Also, $M(g) = g' \in \G$ is again a geodesic. By a result of Nathanson (\cite[Theorem 1]{nathanson}, or simply by computing the real and imaginary parts), $M(\D) \subseteq \D$, so by continuity of $M$ on $\Cbar$, $M(\Dbar) \subseteq \Dbar.$ Hence, $g' \subseteq \Dbar$. Thus, $M(P)$ is an open region contained in $\Dbar$ which is bounded by $M(I)$ and some geodesic $g' \subseteq \Dbar$, which means a slice with endpoints $M(x), M(y), M(x)< M(y)$. 

\end{proof}
\begin{corollary}\label{imageisslice}
Let $M = \mat{a}{b}{c}{d} \in \SLN$. Then: $M(\D)$ is a slice with endpoints $b/d$ and $a/c$ with $b/d < a/c$; moreover, if $c \neq 0$, $\diam(M(P)) < \infty$.
\end{corollary}
\begin{proof}
Note $\D$ is a slice with endpoints $0, \infty$. We have $M(0) = b/d, M(\infty) = a/c$, so by Lemma \ref{everything}, $M(P)$ is a slice with those endpoints. Since $ad - bc  = 1$, $d \neq 0$, so when $c \neq 0$, we have 
$$\diam(M(P)) = a/c - b/d < \infty.$$ 
\end{proof}

Now we prove Theorem \ref{lrpairs}:
\begin{proof}[Proof of Theorem \ref{lrpairs}]
It is clear from the geometric interpretation of slices that two slices $A(\D), B(\D)$ do not intersect if and only if $I(A(\D)), I(B(\D))$ do not intersect on the interior. By Corollary \ref{imageisslice}, this happens if and only if either $a_1/c_1 \leq b_2/d_2$ or $a_2/c_2 \leq b_1/d_1$, which is equivalent to the statement of the theorem. 
\end{proof}

\section{Trees in $\fancyF(L, R)$ are always rooted}
In this section we show $\fancyF(L, R)$ contains only rooted trees. 
We let
\begin{defn}
Let $\L: \R_{> 0} \to \R_{> 0}$ via 
$$\L(t) := \frac{t}{t+1}.$$ Note that for $n \in \N$, 
$$\L^n(t) = \mat{1}{0}{1}{1}^n (t) = \frac{t}{nt + 1},$$ which decreases goes to $0$ as $n \to \infty$ for all $t > 0$. 
\end{defn}
\begin{lemma}\label{lengthchange}
Let $M = \mat{a}{b}{c}{d}  \in \SLN$. Let $P$ be a slice. Write $I:= I(P) = [x, x+t] \subseteq  \R_{\geq 0}$, where $t = \diam(P) < \infty$. Then $M(P)$ is also a slice, and:
\begin{enumerate}
\item If $c \neq 0$, $\diam(M(P)) \leq \L(t)$;
\item If $c = 0$, $\diam(M(P)) = t.$
\end{enumerate}
\end{lemma}
\begin{proof}
If $c = 0$, then $M$ is a translation, and $\diam(M(P)) = \diam(P)$. 

Suppose now $c \neq 0$.
By Lemma \ref{everything}, $M(P)$ is a slice with endpoints $M(x), M(x + t)$. We have:
\begin{align*}
M(x+t) - M(x)  &= \frac{a(x+t) + b}{c(x+t) + d} - \frac{ax + b}{cx + d} \\
&= \frac{\Big( a(x+t) + b\Big)\Big( cx + d\Big)- \Big( ax + b\Big)\Big( c(x+t) + d\Big)}{\Big( cx + d\Big)\Big( c(x+t) + d\Big)}  \\
&=\frac{t(acx + ad - cax - cb)}{\big( cx + d\big)\big( c(x+t) + d\big)} \\
&=\frac{t}{\big( cx + d\big)\big( c(x+t) + d\big)} \\
\end{align*}
 Since $ad - bc = 1$ and by assumption $c \neq 0$, we have $c, d \geq 1$, so
$$\frac{t}{\big( cx + d\big)\big( c(x+t) + d\big)} \leq  \frac{t}{ cx+ct + d} \leq \frac{t}{ t + 1} = \L(t)$$ as desired.
\end{proof}

\begin{theorem}\label{radto0}
Let $M = \mat{a}{b}{c}{d} \in \SLN$ with $c \neq 0$. Then $$\diam(M^n(\D))  \underset{n \to \infty}{\longrightarrow} 0.$$
\end{theorem}
\begin{proof}
By Corollary \ref{imageisslice}, $M(\D)$ is a slice with $\diam(M(\D)) =: d < \infty$. By Lemma \ref{lengthchange}, it follows that $M^n(\D)$ is a slice with $$\diam(M^n(\D)) \leq \L^{n - 1}(d),$$ which goes to $0$ as $n$ goes to infinity.
\end{proof}

\begin{corollary}\label{zerointersect}
For all $M = \mat{a}{b}{c}{d} \in \SLN$ with $M \neq \id$,
$$\bigcap\limits_{n \in \Nz} M^n(\D) = \emptyset.$$
\end{corollary}
\begin{proof}
If $c \neq 0$, this follows immediately from Theorem \ref{radto0}, because for every $z \in \D$, for sufficiently large $n$, we will have $\rad(M^n(\D)) < \Im(z)$. If $c = 0$, $M$ is a translation by a positive integer $b$, so again we are done.
\end{proof}

\begin{lemma}\label{cnot0}
Let $A = \mat{a_1}{b_1}{c_1}{d_1}, B= \mat{a_2}{b_2}{c_2}{d_2}$ such that $(L, R)$ is a left-right pair. Then at least one of $c_1, c_2$ is non-zero.
\end{lemma}
\begin{proof}
Suppose $c_1 = c_2 = 0$. Then $A, B$ are both translations, and clearly $A(\D) \cap B(\D) \neq \emptyset$, which is a contradiction. 
\end{proof}
We are now ready to prove the main theorem. It suffices to show that any vertex $z \in \D$ cannot have infinitely many ancestors. We will prove this by contradiction. In particular, by assuming $z$ has infinitely many ancestors, we will show that $\Im (z)$ has to be less than $\eps$ for any $\eps > 0$. 

\begin{proof}[Proof of Theorem \ref{allrooted}]
\hspace{1in}\\

Write $$L =: \mat{a_1}{b_1}{c_1}{d_1}, R=: \mat{a_2}{b_2}{c_2}{d_2}.$$  By Lemma \ref{cnot0}, without loss of generalty $c_1 \neq 0$. Moreover, $R \neq \id$, because else $L(\D) \cap R(\D) = L(\D) \neq \emptyset$.  Suppose $z \in \D$ has infinitely many ancestors. By Corollary \ref{zerointersect}, any vertex can only have finitely many consecutive right ancestors. Hence, for every $n \in \N$, we can find $w \in \D$ such that  
$$z  = R^{\alpha_{1}} \circ L^{\alpha_{2}} \circ \cdots \circ L^{\alpha_{2n}}(w),$$ where $\alpha_{2k} \geq 1$ and $\alpha_{2k  - 1} \geq 0$ for $k \in \{1, \ldots, n\}$. 
 By Corollary \ref{imageisslice}, $L(\D)$ is a slice with $\diam(L(\D)) =: d < \infty$ and, by Lemma \ref{lengthchange},
 
$$\diam ( R^{\alpha_{1}} \circ L^{\alpha_{2}} \circ \cdots \circ L^{\alpha_{2n} - 1}(L(\D)) \leq \L^{\alpha_2 + \cdots + \alpha_{2n - 2} } (d)\leq \L^{n - 1}(d).$$ 

Since $z$ has to lie in $R^{\alpha_{1}} \circ L^{\alpha_{2}} \circ \cdots \circ L^{\alpha_{2n} - 1}(L(\D))$ for all $n$, it follows that 
$$0<\Im(z)  \leq  \L^{n - 1}(d) \underset{n \to \infty}{\longrightarrow} 0,$$ which is a  contradiction.

\end{proof}
\section{Acknowledgments}
The research was conducted during the Undergraduate Mathematics Research Program at University of Minnesota Duluth, and supported by grants NSF-$1358659$ and NSA H$98230$-$16$-$1$-$0026$. I would like to thank Joe Gallian very much for the wonderful Duluth REU, and for begin such an incredible and supportive mentor. 
\bibliographystyle{plain}
\bibliography{complexbinarytrees}

\end{document}